\documentclass{amsart}
\usepackage[utf8]{inputenc}
\usepackage{amsmath}
\usepackage{amssymb}
\usepackage{amsfonts}

\usepackage{amsthm}

\numberwithin{equation}{section}

\newtheorem{theorem}{Theorem}
\newtheorem*{theorem*}{Theorem}
\newtheorem{lemma}{Lemma}[section]

\newcommand{\F}{\mathcal{F}}
\newcommand{\E}{\mathbb{E}}
\newcommand{\PP}{\mathbb{P}}
\newcommand{\eps}{\varepsilon}
\newcommand{\pphi}{\varphi}
\newcommand{\G}{\mathcal{G}}
\newcommand{\HH}{\mathcal{H}}
\newcommand{\BMO}{\mathrm{BMO}}

\begin{document}
	\title{On multipliers into martingale $SL^\infty$ spaces for arbitrary filtrations}
	\author{Anton Tselishchev}
	\address{Institute of Analysis, Johannes Kepler University Linz, Alternberger Strasse 69, A-4040 Linz, Austria}
	\email{celis-anton@yandex.ru}
	\thanks{This research was supported by the National Science Centre, Poland, and Austrian Science Foundation FWF joint CEUS programme. National Science Centre project no. 2020/02/Y/ST1/00072 and FWF project no. I5231 and no. P344114}
	\keywords{Martingale square functions, correction theorems, conditional square function}
	\subjclass[2020]{60G42, 60G46, 46B10}
	
	\begin{abstract}
		In this paper we study the following problem: for a given bounded positive function $f$ on a filtered probability space can we find another function (a multiplier) $m$, $0\le m\le 1$, such that the function $mf$ is not ``too small'' but its square function is bounded? We explicitly show how to construct such multipliers for the usual martingale square function and for so-called conditional square function. Besides that, we show that for the usual square function more general statement can be obtained by application of a non-constructive abstract correction theorem by S. V. Kislyakov.
	\end{abstract}

	\maketitle
	
	\section{Introduction}
	
	Let $(\Omega, \F, \PP)$ be a (standard) probability space and fix a filtration of sigma-algebras $(\F_n)_{n\ge 0}$ on it. We will always assume (for simplicity) that $\bigcup_{n\ge 0} \F_n$ generates $\F$ and that $\F_0=\{\emptyset, \Omega\}$. Denote by $\E_n$ the operator of conditional expectation with respect to $\F_n$ and by $\Delta_n$ the martingale difference: $\Delta_n=\E_n-\E_{n-1}$, $n\ge 1$. Then the square function operator is defined, say, for $L^2$ functions $g$ on $\Omega$ as
	\begin{equation}
	S(g)=\Big(\sum_{n\ge 1} |\Delta_n g|^2\Big)^{1/2}. \label{S_def}
	\end{equation}
	We will also be interested in a conditional square function operator defined by the following formula:
	\begin{equation}
	\sigma(g)=\Big( \sum_{n\ge 1} \E_{n-1}|\Delta_n g|^2 \Big)^{1/2}.\label{sigma_def}
	\end{equation}
	
	One of the simplest and most well studied examples of a filtration on the probability space is the dyadic filtration. In this case each sigma-algebra $\F_n$ is generated by a finite number of atoms and each atom in $\F_n$ is split into two atoms of equal mass in $\F_{n+1}$. It is not difficult to see that in this case the operators $S$ and $\sigma$ coincide. The properties of the martingale square function operator in the dyadic case are extensively studied and very well known; a systematic treatment of this operator as well as other questions of dyadic harmonic analysis can be found e.g. in \cite[Chapter 1]{Mul_book}.
	
	The situation, however, becomes much more difficult in the case of general filtered probability spaces. A lot of standard concepts of analysis such as square functions, maximal functions, Hardy and $\BMO$ spaces, atomic decompositions, Muckenhoupt weights, etc. have their martingale counterparts in this general setting but their treatment becomes much more involved (see e.g. the book \cite{Garsia}). For instance, the square function operator $S$ is bounded on $L^p$ for every $1 < p < \infty$. However, it is known that in general the operator $\sigma$ is bounded on $L^p$ only if $p\ge 2$ (see \cite[Theorem 2.11 and Example 2.17]{Weisz}). The systematic treatment of the above mentioned concepts in the context of general filtered probability spaces can be found in the book \cite{Long}; see also the book \cite{RevYor} for the case of continuous time filtrations.
	
	It is well known that if $f$ is a bounded function then the functions $S(f)$ and $\sigma(f)$ are not necessarily bounded. The explicit simple example of a bounded function $f$ such that $S(f)$ is unbounded in a particular case of the dyadic filtration on the interval $[0,1]$ is presented in \cite[pp. 67--68]{Mul_book}. It is worth noting that in general the condition ``$S(f)\in L^\infty$'' is very restrictive (see e.g. \cite{ChWilWol}). The space of functions $f\in L^2$ which satisfy this condition is usually denoted by $SL^\infty$. Some properties of functions in this space (in the dyadic case) can be found in \cite[Section 1.3]{Mul_book}.
	
	For the dyadic filtration, the properties of Banach space $SL^\infty$ have been studied in the papers \cite{Lech1} and \cite{Lech2}. However, we would like to point out that it is noted in a recent monograph \cite[p. 86]{Mul_book2} that the Banach space properties of the space $SL^\infty$ associated to a general filtered probability space $(\Omega, (\F_n), \PP)$ are vastly untouched. In this regard, the present paper can be viewed as an attempt to initiate the study of the spaces $SL^\infty$ for arbitrary filtrations.
	
	In the case of dyadic filtration it is known that for each positive function $f$ such that $f\le 1$ a multiplier function $m$ can be found such that $0\le m\le 1$, $\|S(mf)\|_\infty \leq C$ and the mean value of $mf$ is not too small: $\int mf \ge c \int f$. Here $C$ and $c$ are some absolute constants. In the paper \cite{J-M} such function $m$ was explicitly constructed (see also presentation of this construction in \cite[Theorem 1.3.8]{Mul_book}). Using similar ideas for the Brownian motion, such multiplier was also constructed in \cite{J-M} for a Littlewood--Paley square function operator (in this case $f$ is supposed to be a function on a unit circle $\mathbb{T}$). These investigations were motivated by a similar question concerning Cauchy integrals over Lipschitz curves which was posed in \cite{Jones} (this question remains unanswered). 
	
	In view of the above discussion, a natural question arises: is it possible to extend the above result to the general setting of arbitrary (discrete) filtrations? This question is answered in the two following theorems. 
	
	Consider at first the conditional square function operator $\sigma$ defined by the formula \eqref{sigma_def}. We will prove the following statement for it.
	\begin{theorem}
		Let $(\Omega, (\mathcal{F}_n)_{n\ge 0}, \PP)$ be a filtered probability space and suppose that $f$ is a positive function on $\Omega$ such that $f\le 1$. Then there exists a function $m:\Omega\to [0,1]$ such that 
		\begin{equation}
		\E(mf)\ge e^{-\|f\|_{\infty}}\E f
		\end{equation}
		and the function $\sigma(mf)$ is uniformly bounded:
		\begin{equation}
		\sigma(mf)^2\le 18.\label{est_sigma}
		\end{equation}	\label{mult_sigma}
	\end{theorem}

	Moreover, in the proof of this theorem we will \emph{explicitly construct} such function $m$; see next section for details. This construction is a natural generalization of the construction for a dyadic filtration in \cite{J-M}. However, in our general setting of arbitrary filtered probability spaces some extra difficulties arise and therefore new additional ideas are required.
	
	Let us turn now to the case of the square function operator operator $S$ defined by the formula \eqref{S_def}. We have a following similar statement for it.
	
	\begin{theorem}
		Let $(\Omega, (\mathcal{F}_n)_{n\ge 0}, \PP)$ be a filtered probability space and suppose that $f$ is a positive function on $\Omega$ such that $f\le 1$. Then there exists a function $m:\Omega\to [0,1]$ such that 
		\begin{equation}
		\E(mf)\ge e^{-\|f\|_{\infty}}\E f
		\end{equation}
		and the function $S(mf)$ is uniformly bounded:
		\begin{equation}
		S(mf)^2\le C. \label{est_S}
		\end{equation}	\label{mult_S_const}
		Here $C$ is an absolute constant.
	\end{theorem}

	We do not have the goal of providing optimal constants in inequalities \eqref{est_sigma} and \eqref{est_S}. The main feature of these theorems is that we provide the explicit construction of the required multipliers (although, despite the fact that the construction in the proof of Theorem~\ref{mult_S_const} is another generalization of the construction in the dyadic case, it is more complicated than the proof of Theorem~\ref{mult_sigma}: we are able to provide an algorithm for the construction of the desired multiplier function $m$ only if all sigma-algebras $\F_n$ are generated by finite number of atoms; this issue is discussed in Subsection~3.2 below).
	
	Several remarks are in order. In the paper \cite{J-M} a similar question was studied for the dyadic filtration which is a simplest example of a discrete filtration on a probability space. The proof in \cite{J-M} heavily relies on the dyadic structure and in particular on the fact that in this case the function $|\Delta_n (f)|^2$ is $\F_{n-1}$-measurable. Therefore, it seems rather surprising that the generalization to the case of arbitrary filtration is possible for both operators $S$ and $\sigma$.
	
	In \cite{J-M} a similar question was also studied for (continuous time) filtration associated to the Brownian motion. However, in a certain sense this can be viewed as another ``extreme case'' because the construction of the corresponding multiplier is naturally dictated by the Ito's formula. As we mentioned above, using this result, in the same paper the multiplier for the case of Littlewood--Paley square function on the torus was also constructed. However, the resulting way of constructing such multiplier turns out to be rather difficult since it involves a detour through the field of Stochastic Analysis. It is possible that the results for arbitrary discrete filtrations which we prove in this paper can lead to a more direct construction of aforementioned multiplier for the Littlewood--Paley square function which would not use Brownian motion.
	
	There is another way of treating the question of existence of such multipliers: more precisely, it falls into the framework of ``correction theorems'' --- subject developed mainly by S. V. Kislyakov in the papers \cite{Kis1, Kis2, Kis, Anis-Kis, IvKis, KisPerst}. In a broad sense, correction theorems are the statements of the following form: a certain (``bad'') function can be slightly modified so that the resulting function has some ``good'' properties. More specifically, in the papers \cite{Anis-Kis, Kis} several theorems of the following kind were proved: if $f$ is a function in $L^\infty(\mathbb{R}^d)$ such that $\|f\|_{\infty} \le 1$ then we can find a function $g$ such that it differs from $f$ only on a set of small measure and $Tg\in L^\infty(\mathbb{R}^d)$; here $T$ stands for a singular integral operator (these operators are typically not bounded on $L^\infty(\mathbb{R}^d)$).
	
	Moreover, in the paper \cite{Kis} a general abstract correction theorem was proved and the statements of the above kind were derived from it. We will present the statement of this general theorem and show that it is applicable to our problem in Section~4. Now we just state the resulting theorem for operator $S$.
	
	\begin{theorem}
		Let $f$ be a function on $\Omega$ such that $\|f\|_\infty \le 1$. Then for every $0<\eps\le 1$ there exists a function $g$ on $\Omega$ which satisfies the following conditions:
		\begin{gather}
		|g|+|f-g|=|f|\label{abs_mult_cond};\\
		\PP\{f\neq g\}\le \eps\|f\|_1\label{abs_norm_bound};\\
		\|S(g)\|_\infty \le C(1+\log \eps^{-1})\label{logar}.
		\end{gather}
	Here $C$ is some absolute constant.
	\label{mult_S}
	\end{theorem}

	We note that for positive functions $f$ the condition \eqref{abs_mult_cond} means that $g$ is obtained from $f$ by multiplication by some function $m$ between $0$ and $1$. Moreover, it can be easily seen that the condition \eqref{abs_norm_bound} implies that $\E g\ge (1-\eps)\E f$. Therefore, this theorem answers our initial question for the operator $S$ (and it is in fact more general since the condition \eqref{abs_norm_bound} implies that $m=1$ on a set of large measure).
	
	However, this theorem has a drawback: its proof is non-constructive since it relies on Kislyakov's abstract correction theorem from the paper \cite{Kis} which in turn exploits Hahn--Banach theorem and an ingenious duality argument. In view of the previous discussion, we could say that our proofs of Theorems \ref{mult_sigma} and \ref{mult_S_const} represent an example where the abstract arguments in the context of correction theorems can be avoided. The price for the ``constructivity'' is that the resulting statements are slightly less powerful because we cannot guarantee that the multiplier function $m$ is equal to $1$ on the set of large measure (note, however, that a logarithmic estimate similar to \eqref{abs_norm_bound} and \eqref{logar} can be deduced from Theorems 1 and 2 by iterations). On the other hand, it is worth noting that we do not know if there exists an abstract approach which would be applicable to prove a correction theorem for the operator~$\sigma$.
	
	The next three sections are devoted to the proofs of the formulated theorems. In the final section we discuss certain open problems.
	
	\section{Proof of theorem \ref{mult_sigma}: explicit construction}
	
	\subsection{Outline of proof}
	
	Our construction of the multiplier $m$ is essentially a (far-reaching) generalization of the construction from the paper \cite{J-M}. The function $m$ will be of the form $m=e^\pphi$, where the formula for $\pphi$ is 
	$$
	\pphi=f-\|f\|_\infty -\sum_{n\ge 1} \alpha_n \E_{n-1}|\Delta_n f|^2=-\|f\|_\infty +\E f + \sum_{n\ge 1} (\Delta_n f - \alpha_n\E_{n-1}|\Delta_n f|^2).
	$$
	Here $\alpha_n$ are certain positive $\F_{n-1}$-measurable functions. The convergence of series in this formula is understood in the almost everywhere sense (the sum might be equal to $-\infty$ at certain points).
	
	We will choose the functions $\alpha_n$ so that the following conditions hold:
	\begin{gather}
		\E_{n-1}\Big[f\cdot e^{\Delta_n f - \alpha_n\E_{n-1}|\Delta_n f|^2}\Big] =\E_{n-1} f\label{main-ident};\\
		\label{main-ineq}\alpha_n \ge 1/2 \quad \text{a.e.}
	\end{gather}

	We postpone the details concerning the choice of such functions $\alpha_n$ till the next subsection and assume for now that we have chosen $\alpha_n$ satisfying these properties. We need to show that
	$$
	\sigma (e^\pphi f)^2=\sum_{n\ge 1} \E_{n-1}|\Delta_n (e^\pphi f)|^2 \leq 18\quad\text{a.e.}
	$$
	
	Let us fix one number $n\ge 1$ and consider the expression $\E_{n-1}|\Delta_n (e^\pphi f)|^2$. We can write the following formula for the function $\pphi$:
	$$
	\pphi = -\|f\|_\infty +\E f + v_n + (\Delta_n f - \alpha_n \E_{n-1}|\Delta_n f|^2) + z_n,
	$$
	where
	\begin{gather}
		v_n=\sum_{m<n} (\Delta_m f - \alpha_m\E_{m-1}|\Delta_m f|^2);\label{vn}\\
		z_n=\sum_{m>n} (\Delta_m f - \alpha_m\E_{m-1}|\Delta_m f|^2).\label{zn}
	\end{gather}
	Obviously, $v_n$ is an $\F_{n-1}$-measurable function.
	
	Therefore, we have:
	$$
	\Delta_n(e^\pphi f)=e^{-\|f\|_\infty + \E f + v_n-\alpha_n \E_{n-1}|\Delta_n f|^2}\Delta_n( e^{z_n} e^{\Delta_n f} f).
	$$
	Our next goal is to get rid of the factor $e^{z_n}$ inside the martingale difference. In order to do that, we need the following lemma.
	
	\begin{lemma}
		\label{induct}
		We have the following identity:
		$$
		\E_n(e^{z_n} f) = \E_n f.
		$$
	\end{lemma}

	\begin{proof}
		The proof consists of two steps: at first we suppose that the sum in the formula \eqref{zn} defining the function $z_n$ is finite and prove the lemma under this assumption and then we will pass to a limit.
		
		For each number $M>n$ denote
		$$
		z_{n,M}=\sum_{m=n+1}^M (\Delta_m f - \alpha_m\E_{m-1}|\Delta_m f|^2).
		$$
		The function $z_{n,M}$ is $\F_M$-measurable. 
		
		Now we can write the following (trivial) identity:
		$$
		\E_n(e^{z_{n,M}f})=\E_n\circ \E_{n+1}\circ\ldots\circ \E_{M-1}(e^{z_{n,M}}f).
		$$
		This identity lets us ``isolate'' one summand in the formula for $z_{n,M}$ and apply the equality \eqref{main-ident}. Indeed, using our measurability assumptions and the basic properties of conditional expectation we write:
		\begin{multline*}
			\E_n(e^{z_{n,M}} f) = \E_n\circ\ldots\circ\E_{M-2}\Big[e^{z_{n,M-1}}\E_{M-1}(f\cdot e^{\Delta_M f - \alpha_M\E_{M-1}|\Delta_M f|^2})\Big]\\
			=\E_n\circ\ldots\circ\E_{M-2}\Big[e^{z_{n,M-1}}\E_{M-1} f\Big]=\E_n\circ\ldots\circ\E_{M-2}\circ\E_{M-1}(e^{z_{n,M-1}f})\\=\E_n\circ\ldots\circ\E_{M-2}(e^{z_{n,M-1}}f).
		\end{multline*}
	The equality \eqref{main-ident} was used in order to pass from the first line to the second.
	
	Now we see that arguing like this we can prove that 
	$$
	\E_n(e^{z_{n,M}} f)=\E_n f.
	$$
	It remains only to pass to the limit as $M\to\infty$ which is possible since $e^{z_{n,M}}\to e^{z_n}$
	a.e. and therefore in $L^1$ (and $\E_n$ is a bounded linear operator on $L^1$).
	\end{proof}

	Using this lemma, we can write:
	\begin{multline*}
	\Delta_n(e^{z_n}e^{\Delta_n f}f)=\Delta_n\circ\E_n(e^{z_n}e^{\Delta_n f}f)=\Delta_n(e^{\Delta_n f}\E_n(e^{z_n} f))\\=\Delta_n(e^{\Delta_n f}\E_n f)=\Delta_n(e^{\Delta_n f} f).
	\end{multline*}

	Hence we have:
	\begin{equation}
	\E_{n-1}|\Delta_n (e^\pphi f)|^2=e^{2(-\|f\|_\infty+\E f + v_n-\alpha_n\E_{n-1}|\Delta_n f|^2)}\E_{n-1}|\Delta_n(e^{\Delta_n f} f)|^2.\label{formula}
	\end{equation}
	Recall the formula \eqref{vn}. We obviously have
	$$
	-\|f\|_\infty + \E f + \sum_{m<n}\Delta_m f = -\|f\|_\infty + \E_{n-1}f \leq 0
	$$
	and therefore the right-hand side of the formula \eqref{formula} is not greater than
	\begin{equation}
	e^{-2\sum_{m=1}^n \alpha_m \E_{m-1}|\Delta_m f|^2}\E_{n-1}|\Delta_n (e^{\Delta_n f} f)|^2.\label{expression}
	\end{equation}
	Now we would like to use the following inequality.
	\begin{lemma}
		For any function $f\in L^\infty$ such that $0\le f\le 1$ we have
		$$
		\E_{n-1}|\Delta_n(e^{\Delta_n f} f)|^2 \leq 18 \E_{n-1}|\Delta_n f|^2.
		$$\label{main_lemma}
	\end{lemma}

	We will prove this lemma in Subsection 2.3. Assume for now that it is true. Then, using it and the fact that $\alpha_n\ge 1/2$ we finally get that the expression \eqref{expression} is bounded by
	$$
	18 e^{-\sum_{m=1}^n\E_{m-1}|\Delta_m f|^2}\E_{n-1}|\Delta_n f|^2.
	$$
	Let us fix one point $t$ and denote $q_n=(\E_{n-1}|\Delta_n f|^2) (t)$. Then we proved that 
	$$
	\sigma(e^\pphi f)^2 (t)\leq 18 \sum_{n\ge 1}e^{-\sum_{m=1}^n q_m} q_n.
	$$
	The sum in the right hand side of this formula can be estimated by appropriate integral:
	$$
	18\sum_{n\ge 1}e^{-\sum_{m=1}^n q_m} q_n\leq 18\int_0^{+\infty} e^{-x} dx=18.
	$$
	
	It remains only to show that $\E(e^\pphi f)$ is not too small. But arguing as before (again, using equation \eqref{main-ident}) we get that
	$$
	\E (e^\pphi f)=e^{-\|f\|_{\infty} + \E f}\E f \ge e^{-\|f\|_\infty}\E f.
	$$
	
	Theorem~\ref{mult_sigma} is now proved if we show that we can choose the functions $\alpha_n$ satisfying conditions \eqref{main-ident} and \eqref{main-ineq} and prove Lemma~\ref{main_lemma}. These are the subjects of the next two subsections.
	
	\subsection{Construction of functions $\alpha_n$}
	
	Note that if $\alpha_n$ is an $\F_{n-1}$-measurable function then it is uniquely determined by the identity \eqref{main-ident}. Indeed, this identity is equivalent to the following:
	$$
	\E_{n-1}(f\cdot e^{\Delta_n f})=e^{\alpha_n\E_{n-1}|\Delta_n f|^2}\E_{n-1}f	$$
	and therefore we can write the following explicit formula for the function $\alpha_n$:
	$$
	\alpha_n=\frac{1}{\E_{n-1}|\Delta_n f|^2}\log\Big(\frac{\E_{n-1}(f\cdot e^{\Delta_n f})}{\E_{n-1} f}\Big).	
	$$
	On the $\F_{n-1}$-measurable set where one of the expressions $\E_{n-1}|\Delta_n f|^2$ and $\E_{n-1} f$ turns to zero the value of $\alpha_n$ can be arbitrary.
	
	It remains only to check that $\alpha_n\ge 1/2$, that is, to prove the inequality 
	\begin{equation}
	\E_{n-1}(f\cdot e^{\Delta_n f})\ge e^{\frac{1}{2}\E_{n-1}|\Delta_n f|^2}\E_{n-1}f.\label{main-est-sigma}
	\end{equation}
	
	It will be the consequence of the following inequality:
	$$
	\E_{n-1}(f\cdot e^{\Delta_n f})\ge (1+\E_{n-1}|\Delta_n f|^2)\E_{n-1}f.
	$$
	Indeed, here the right-hand side is bigger than the right-hand side of \eqref{main-est-sigma} because $|\Delta_n f|\le 1$ and $e^x\le 1+2x$ if $0\le x\le 1/2$.
	
	Let us use the Taylor expansion for the exponential function:
	\begin{equation}
	e^{\Delta_n f}=1+\Delta_n f+ \frac{1}{2}|\Delta_n f|^2 + Q_n,\label{taylor_exp}
	\end{equation}
	where $Q_n$ is $\F_n$-measurable function such that $|Q_n|\le \frac{1}{2}|\Delta_n f|^3$ (this is a consequence of one more elementary estimate: $|e^x-1-x-x^2/2|\le |x|^3/2$ for $x\le 1$).
	
	Now we can rewrite the desired estimate:
	$$
	\E_{n-1}[f\cdot (1+\Delta_n f + |\Delta_n f|^2/2 + Q_n)]\ge (1+\E_{n-1}|\Delta_n f|^2)\E_{n-1} f
	$$
	which is equivalent to the inequality
	$$
	\E_{n-1}(f\Delta_n f)+\frac{1}{2}\E_{n-1}(f|\Delta_n f|^2)+\E_{n-1}(fQ_n)\ge \E_{n-1}|\Delta_n f|^2 \E_{n-1} f.
	$$
	This estimate is easy. Indeed, it is not difficult to check that $\E_{n-1}(f\Delta_n f)=\E_{n-1}|\Delta_n f|^2\ge \E_{n-1}|\Delta_n f|^2 \E_{n-1} f$ and since $Q_n\ge -\frac{1}{2}|\Delta_n f|^3\ge -\frac{1}{2}|\Delta_n f|^2$ we also see that
	$$
	\frac{1}{2}\E_{n-1}(f|\Delta_n f|^2)+\E_{n-1}(fQ_n)\ge 0.
	$$
	
	\subsection{Proof of Lemma~\ref{main_lemma}}
	
	Finally, we only need to prove the inequality
	$$
	\E_{n-1}|\Delta_n(e^{\Delta_n f} f)|^2 \leq 18 \E_{n-1}|\Delta_n f|^2.
	$$
	Once again we use the Taylor expansion:
	$$
	e^{\Delta_n f}=1+\Delta_n f +R_n,
	$$
	where $R_n$ is an $\F_n$-measurable function such that $|R_n|\le |\Delta_n f|^2\le 1$. This is the consequence of elementary inequality $|e^x-1-x|\le x^2$ for $x\in [-1,1]$.
	
	Let us estimate the left-hand side of the desired inequality in a following way:
	\begin{multline*}
		\E_{n-1}|\Delta_n(e^{\Delta_n f} f)|^2=\E_{n-1}|\Delta_n[(1+\Delta_n f + R_n)f]|^2\\
		=\E_{n-1}|\Delta_n [f+(\E_nf-\E_{n-1}f)f + R_nf]|^2\\
		=\E_{n-1}|(1-\E_{n-1}f)\Delta_n f + \Delta_n (f\cdot\E_n f)+\Delta_n(R_n f)|^2\\
		\leq 3\big(\E_{n-1}|(1-\E_{n-1}f)\Delta_n f|^2 + \E_{n-1}|\Delta_n (f\cdot\E_n f)|^2 + \E_{n-1}|\Delta_n(R_n f)|^2\big).
	\end{multline*}
	We estimate each of three summands separately. At first, we obviously have
	$$
	\E_{n-1}|(1-\E_{n-1}f)\Delta_n f|^2 \leq \E_{n-1}|\Delta_n f|^2.
	$$
	The estimate for the third summand is also easy. We use that for an arbitrary (bounded) function $g$ we have $\E_{n-1}|\Delta_n g|^2=\E_{n-1}|\E_n g|^2 - |\E_{n-1} g|^2\leq \E_{n-1}|\E_n g|^2$ and write:
	\begin{multline*}
	\E_{n-1}|\Delta_n (R_nf)|^2\leq \E_{n-1}|\E_n(R_nf)|^2\leq \E_{n-1}(\E_n|R_n f|)^2\leq \E_{n-1} (\E_n |R_n f|)\\
	=\E_{n-1}|R_n f|\leq \E_{n-1}|\Delta_n f|^2.
	\end{multline*}

	It remains now only to prove that
	$$
	\E_{n-1}|\Delta_n(f\cdot\E_nf)|^2\leq 4 \E_{n-1}|\Delta_n f|^2.
	$$
	Let us denote $g=\E_n f$. Then $g$ is a function between $0$ and $1$ and we need to show that
	$$
	\E_{n-1}|\Delta_n g^2|^2 \leq 4\E_{n-1}|\Delta_n g|^2.
	$$
	This inequality is easy to prove if $n=0$ or if $\F_n$ and $\F_{n-1}$ are purely atomic. However, some work is required in order to prove it in a general case. First of all, it is better to rewrite both sides using $\F_n$-measurability of $g$:
	$$
	\E_{n-1}|\Delta_n g|^2 = \E_{n-1}\big[ (\E_n g) ^2 - 2\E_n g\cdot \E_{n-1} g + (\E_{n-1}g)^2\big]=\E_{n-1}g^2- (\E_{n-1}g)^2
	$$
	and similarly
	$$
	\E_{n-1}|\Delta_n g^2|^2 = \E_{n-1} g^4 - (\E_{n-1} g^2)^2.
	$$

	Next, our pointwise inequality is equivalent to the following inequality for integrals: for arbitrary $\F_{n-1}$-measurable set $B$ (of positive measure) we have to prove that
	\begin{equation}
		\label{to_prove_fin}
	\int_B(\E_{n-1}g^4)\, d\PP - \int_B(\E_{n-1}g^2)^2\, d\PP \leq 4\Big[ \int_B (\E_{n-1}g^2)\, d\PP - \int_B (\E_{n-1}g)^2\, d\PP \Big].
	\end{equation}
	
	Let us take an arbitrary $\F_{n-1}$-measurable set $A$ and write the following easy estimate:
	$$
	\frac{1}{\PP(A)^2}\int_A \int_A (g(x)^2-g(y)^2)^2 \, d\PP(x)\, d\PP(y)\\\leq \frac{4}{\PP(A)^2}\int_A \int_A (g(x)-g(y))^2 \, d\PP(x)\, d\PP(y).
	$$
	This is true since $|g(x)^2-g(y)^2|=|(g(x)-g(y))(g(x)+g(y))|\leq 2|g(x)-g(y)|$. The right-hand side can now be rewritten as
	\begin{multline*}
		\frac{1}{\PP(A)^2}\int_A \int_A (g(x)-g(y))^2 \, d\PP(x)\, d\PP(y) = 2\Big[ \frac{1}{\PP(A)} \int_A g^2\, d\PP - \Big(\frac{1}{\PP(A)}\int_A g\, d\PP \Big)^2 \Big]\\
		= 2\Big[ \frac{1}{\PP(A)} \int_A (\E_{n-1} g^2)\, d\PP - \Big(\frac{1}{\PP(A)}\int_A (\E_{n-1} g)\, d\PP \Big)^2 \Big].
	\end{multline*}

	Rewriting the left-hand side in a similar way, we see that we have proved that
	\begin{equation}
		\begin{aligned}
		\frac{1}{\PP(A)} \int_A (\E_{n-1} g^4)\, d\PP - \Big(\frac{1}{\PP(A)}\int_A (\E_{n-1} g^2)\, d\PP \Big)^2 \\
	\leq 4\Big[ \frac{1}{\PP(A)} \int_A (\E_{n-1} g^2)\, d\PP - \Big(\frac{1}{\PP(A)}\int_A (\E_{n-1} g)\, d\PP \Big)^2 \Big].\label{square_diff}
	\end{aligned}
	\end{equation}
	
	This inequality does not imply \eqref{to_prove_fin} immediately. However, on sets where expressions $\E_{n-1} g$ and $\E_{n-1} g^2$ ``do not oscillate much'' they are almost the same. To be more precise, fix some large number $N$ and for $0\le k,l \le N$ set
	$$
	A_{k,l}=\Big\{\frac{k}{N} \le \E_{n-1}g^2< \frac{k+1}{N}, \frac{l}{N} \le \E_{n-1} g < \frac{l+1}{N}\Big\}\cap B.
	$$
	These sets are $\F_{n-1}$-measurable. Some of them have measure $0$ (for example, if $k>l$ then $\PP(A_{k,l})=0$) but most importantly they do not intersect and form the partition of the set $B$. Besides that, for the sets $A_{k,l}$ such that $\PP(A_{k,l})\neq 0$ we have:
	$$
	\Big| \Big( \frac{1}{\PP(A_{k,l})} \int_{A_{k,l}} (\E_{n-1}g) \, d\PP \Big)^2 - \frac{1}{\PP(A_{k,l})} \int_{A_{k,l}} (\E_{n-1} g)^2\, d\PP \Big| \leq \Big( \frac{l+1}{N} \Big)^2 - \Big( \frac{l}{N} \Big)^2 \leq\frac{3}{N}.
	$$
	Similarly,
	$$
	\Big| \Big( \frac{1}{\PP(A_{k,l})} \int_{A_{k,l}} (\E_{n-1}g^2) \, d\PP \Big)^2 - \frac{1}{\PP(A_{k,l})} \int_{A_{k,l}} (\E_{n-1} g^2)^2\, d\PP \Big| \leq \frac{3}{N}.
	$$
	
	Let us apply the inequality \eqref{square_diff} for each set $A_{k,l}$ of non-zero measure instead of $A$. Adding and subtracting the appropriate term to each side, we write it in a following form:
	\begin{multline*}
		\frac{1}{\PP(A_{k,l})}\int_{A_{k,l}} (\E_{n-1}g^4)\, d\PP-\frac{1}{\PP(A_{k,l})}\int_{A_{k,l}} (\E_{n-1} g^2)^2\, d\PP\\ + \Big( \frac{1}{\PP(A_{k,l})}\int_{A_{k,l}} (\E_{n-1} g^2)^2\, d\PP - \Big( \frac{1}{\PP(A_{k,l})} \int_{A_{k,l}}(\E_{n-1}g^2)\, d\PP \Big)^2 \Big)\\
		\leq 4\Big[ \frac{1}{\PP(A_{k,l})}\int_{A_{k,l}} (\E_{n-1}g^2)\, d\PP-\frac{1}{\PP(A_{k,l})}\int_{A_{k,l}} (\E_{n-1} g)^2\, d\PP\\ + \Big( \frac{1}{\PP(A_{k,l})}\int_{A_{k,l}} (\E_{n-1} g)^2\, d\PP - \Big( \frac{1}{\PP(A_{k,l})} \int_{A_{k,l}}(\E_{n-1}g)\, d\PP \Big)^2 \Big) \Big].
	\end{multline*}

	As we have gust explained, the expressions on the second and on the fourth lines of this inequality are small. Hence we see that
	\begin{multline*}
		\frac{1}{\PP(A_{k,l})}\int_{A_{k,l}} (\E_{n-1}g^4)\, d\PP-\frac{1}{\PP(A_{k,l})}\int_{A_{k,l}} (\E_{n-1} g^2)^2\, d\PP\\ \le
		4\Big[ \frac{1}{\PP(A_{k,l})}\int_{A_{k,l}} (\E_{n-1}g^2)\, d\PP-\frac{1}{\PP(A_{k,l})}\int_{A_{k,l}} (\E_{n-1} g)^2\, d\PP\Big] + \frac{15}{N}. 
	\end{multline*}

	It remains to multiply both sides of this inequality by $\PP(A_{k,l})$ and sum over all sets $A_{k,l}$ of non-zero measure. We will get exactly the inequality \eqref{to_prove_fin} with the term $\frac{15}{N}\PP(B)$ added to its right-hand side. However, since $N$ is an arbitrary number, it implies the desired inequality and the proof of Theorem~\ref{mult_sigma} is finished.
	
	\section{Proof of theorem~\ref{mult_S_const}: another version of explicit construction}
	
	\subsection{Outline of proof}
	
	Let us now describe the construction of the multiplier that satisfies conditions of Theorem~\ref{mult_S_const}. The main idea is rather similar to what we have done in the previous section, however certain details need to be significantly changed since now instead of the expressions $\E_{n-1}|\Delta_n f|^2$ we are working with $|\Delta_n f|^2$ which are not $\F_{n-1}$-measurable. We will omit some details in the parts of the proof which are similar to the previous ones. 
	
	Let us put $m=e^\pphi$, where this time
	$$
	\pphi=-\|f\|_{\infty}+f-\sum_{n\ge 1}\beta_n|\Delta_n f|^2=-\|f\|_\infty +\E f + \sum_{n\ge 1}(\Delta_n f-\beta_n |\Delta_n f|^2),
	$$
	where $\beta_n$ are certain positive $\F_n$-measurable functions. The convergence of series here is understood in a.e. sense.
	
	We impose the following conditions on the functions $\beta_n$:
	\begin{gather}
		\E_{n-1}\Big[ f\cdot e^{\Delta_n f - \beta_n |\Delta_n f|^2} \Big]=\E_{n-1} f;\label{main_assump}\\
		\frac{2}{5}\le \beta_n \le \frac{1}{|\Delta_n f|}\quad \text{a.e.} \label{main_beta}
	\end{gather}
	
	We will prove that such functions exist in the next subsection. Assume for now that the functions $\beta_n$ are constructed. We need to estimate the quantity
	$$
	S(e^\pphi f)^2 = \sum_{n\ge 1} |\Delta_n(e^\pphi f)|^2.
	$$
	Fix one number $n$. Then $\pphi$ can be written as
	$$
	\pphi=-\|f\|_\infty + \E f + v_n + (\Delta_n f - \beta_n |\Delta_n f|^2)+z_n,
	$$
	where
	\begin{gather}
		v_n=\sum_{m<n} (\Delta_m f - \beta_m |\Delta_m f|^2);\\
		z_n=\sum_{m>n} (\Delta_m f - \beta_m |\Delta_m f|^2).
	\end{gather}
	We see that $v_n$ is an $\F_{n-1}$-measurable function and hence we have:
	$$
	\Delta_n (e^\pphi f) = e^{-\|f\|_{\infty} + \E f + v_n}\Delta_n (e^{z_n}e^{\Delta_n f - \beta_n |\Delta_n f|^2} f).
	$$
	We can get rid of the factor $e^{z_n}$ inside the martingale difference because as previously we have
	$$
	\E_{n}(e^{z_n} f)=\E_n f.
	$$
	This equality can be proved in exactly the same way as lemma~\ref{induct}.
	
	Therefore, the formula for one martingale difference can be written as
	$$
	\Delta_n (e^\pphi f) = e^{-\|f\|_{\infty} + \E f + v_n}\Delta_n (e^{\Delta_n f - \beta_n |\Delta_n f|^2} f).
	$$
	Since $-\|f\|_\infty+\E f + \sum_{m<n} \Delta_m f \leq 0$ a.e., the square of the martingale difference can be estimated as follows:
	$$
	|\Delta_n (e^\pphi f)|^2 \leq e^{-2\sum_{m=1}^{n-1} \beta_m |\Delta_m f|^2} |\Delta_n (e^{\Delta_n f - \beta_n |\Delta_n f|^2} f)|^2.
	$$
	We would like to prove now that
	\begin{equation}
	|\Delta_n (e^{\Delta_n f- \beta_n |\Delta_n f|^2} f)|\leq (4+e) |\Delta_n f|.\label{main_ineq}
	\end{equation}
	We will prove this inequality in Subsection~3.3. Assuming that it is true, we obtain the following estimate:
	\begin{multline*}
	|\Delta_n (e^\pphi f)|^2 \leq (4+e)^2 e^{-2\sum_{m=1}^{n-1} \beta_m |\Delta_m f|^2} |\Delta_n  f|^2\leq (4+e)^2 e^{-\frac{4}{5}\sum_{m=1}^{n-1} |\Delta_m f|^2} |\Delta_n  f|^2\\
	\leq  (4+e)^2 e^{4/5} e^{-\frac{4}{5}\sum_{m=1}^{n} |\Delta_m f|^2} |\Delta_n  f|^2.
	\end{multline*}

	If we fix one point $t\in \Omega$ and set $p_n=|(\Delta_n f)(t)|^2$, then we see that
	$$
	S(e^\pphi f)^2(t) \leq (4+e)^2 e^{4/5} \sum_{n\ge 1} e^{-\frac{4}{5}\sum_{m=1}^n p_m}p_n.
	$$
	This sum can be estimated by the appropriate integral:
	$$
	\sum_{n\ge 1} e^{-\frac{4}{5}\sum_{m=1}^n p_m}p_n \leq \int_{0}^{+\infty}e^{-4x/5}\, dx=\frac{5}{4}.
	$$
	It gives us the estimate for the square function from Theorem~\ref{mult_S_const}. Besides that, exactly as in the proof of the previous theorem we get that
	$$
	\E (e^\pphi f)=e^{-\|f\|_{\infty} + \E f}\E f \ge e^{-\|f\|_{\infty}}\E f.
	$$
	
	In the next two subsections we show how to construct the functions $\beta_n$ and prove the inequality \eqref{main_ineq}.
	
	\subsection{Construction of functions $\beta_n$}
	
	Let us now address the question of how to construct the functions $\beta_n$ that satisfy conditions \eqref{main_assump} and \eqref{main_beta}. This is not as straightforward as it was in the proof of the previous theorem with the functions $\alpha_n$ since now the equation \eqref{main_assump} is not enough to uniquely determine $\beta_n$. Note that, however, if we put $\beta_n=\frac{1}{|\Delta_n f|}$ then the left-hand side of \eqref{main_assump} will be not greater than the right-hand side. On the other hand, if we put $\beta_n=2/5$ then the left-hand side will be greater than the right-hand side:
	\begin{equation}
		\E_{n-1}\Big[ f\cdot e^{\Delta_n f - \frac{2}{5} |\Delta_n f|^2} \Big]\ge \E_{n-1} f.\label{toprove}
	\end{equation}
	We will prove this inequality in a moment. For now we would like to mention that if all sigma-algebras $\F_n$ were \emph{purely atomic} then this would be enough to \emph{construct} the functions $\beta_n$. Indeed, we need only to construct $\beta_n$ on each atom $I$ of filtration $\F_{n-1}$ which is divided into the atoms $I_1, \ldots, I_k$ of filtration $\F_n$. At first consider $\beta_n^{(0)}=2/5$ on all atoms $I_1, \ldots, I_k$ and the function $\beta_n^{(1)}$ which is equal to $1/|\Delta_n f|$  on $I_1$ and $2/5$ on all other atoms (if $\Delta_n f = 0$ on $I_1$, then the value of $\beta_n^{(1)}$ is not important; we can put it also equal to $2/5$ for example). The inequality \eqref{toprove} implies that the average of the function $f\cdot e^{\Delta_n f - \beta_n^{(0)}|\Delta_n f|^2}$ over $I$ is not smaller than the average of $f$; if these two averages coincide, then we are done.   If the average of the function $f\cdot e^{\Delta_n f - \beta_n^{(1)}|\Delta_n f|^2}$ over $I$ is still greater that the average of $f$, then we proceed to the next atom: consider the function $\beta_n^{(2)}$ which is equal to $1/|\Delta_n f|$ on $I_1\cup I_2$ and to $2/5$ on the rest of $I$, etc. This process will stop when we are on some atom $I_l$ and then we may put $\beta_n=\beta_n^{(l)}$ on all atoms except $I_l$ and the value of $\beta_n$ on $I_l$ is uniquely determined by the identity \eqref{main_assump} (it is possible to write an exact formula for it).
	
	However, if we want to consider arbitrary filtrations, we need the following lemma.
	
	\begin{lemma}
		Let $(\Omega, \G, \PP)$ be a probability space. Suppose that $\HH\subset\G$ is a sub-sigma-algebra which is countably generated. Let $F$ and $G$ be two functions in $L^\infty(\Omega, \G, \PP)$ such that $F\ge G$, $\E(F|\HH)\ge 0$ and $\E(G|\HH)\le 0$ \emph{(}all inequalities hold almost everywhere\emph{)}. Then there exists a function $H$ in $L^\infty(\Omega, \G, \PP)$ such that $F\ge H \ge G$ and $\E(H|\HH)=0$ \emph{(}also a.e.\emph{)}
	\end{lemma}

	This lemma indeed implies existence of the required functions $\beta_n$: we need to apply it with $\G=\F_{n}$, $\HH=\F_{n-1}$, $F=( e^{\Delta_n f - \frac{2}{5} |\Delta_n f|^2}-1)\E_n f$ and $G=( e^{\Delta_n f - |\Delta_n f|}-1)\E_n f$; the formula for $\beta_n$ then will be as follows:
	$$
	\beta_n=\frac{-\log\Big( \frac{H+\E_n f}{\E_n f} \Big)+\Delta_n f}{|\Delta_n f|^2}.
	$$
	It remains only to note that without loss of generality we may assume that all sigma-algebras $\F_n$ are countably generated since each sigma-algebra generated by a random variable always has this property.

	\begin{proof}
		Let us find a countable algebra of sets $\{A_k\}_{k\ge 1}$ which generates sigma-algebra $\HH$. Fix a number $n$. We can find a function $H_n$ such that $\int_{A_k} H_n \, d\PP=0$ for $k=1,\ldots n$ and $F\ge H_k\ge G$. Indeed, we can find a finite number of pairwise disjoint $\HH$-measurable sets $B_1,\ldots, B_N$ so that for every $1\le k\le n$ each $A_k$ is a union of several sets $B_j$. We need to define the function $H_n$ on each of the sets $B_j$. For $0\le s\le 1$ put $H_{n,s}^{(j)}=sF+(1-s)G$ on $B_j$. Since $\int_{B_j} F \, d\PP = \int_{B_j} \E(F|\HH)\, d\PP \ge 0$ and $ \int_{B_j} G\, d\PP = \int_{B_j}\E(G|\HH)\le 0$, we see that $\int_{B_j}H_{n,0}^{(j)}\, d\PP \le 0$ and $\int_{B_j}H_{n,1}^{(j)}\, d\PP \ge 0$. Therefore we can find such $s_j$ that $\int_{B_j}H_{n,s_j}^{(j)}\, d\PP =0$ and put $H_n= H_{n,s_j}^{(j)}$ on $B_j$. Then $\int_{B_j} H_n\, d\PP = 0$ for every $j$ and hence $\int_{A_k} H_n\, d\PP = 0$ for $k=1,\ldots, n$.
		
		The functions $H_n$ lie in a ball of the space $L^\infty (\Omega,\G, \PP)$ and hence we can define $H$ as a weak* limit of a subsequence of $H_n$. Then we have $F\ge H\ge G$ a.e. and $\int_{A_k} H\, d\PP = \int_{A_k}\E(H|\HH)\, d\PP= 0$ for every $k$. It means that $
		\E(H|\HH) = 0$.
	\end{proof}

	It remains only to prove the inequality \eqref{toprove}. We can consider the Taylor expansion of the exponential function \eqref{taylor_exp} and also use the inequality $e^{-\frac{2}{5}|\Delta_n f|^2}\ge 1-\frac{2}{5}|\Delta_n f|^2$ in order to conclude that
	$$
		\E_{n-1}\Big[ f\cdot e^{\Delta_n f - \frac{2}{5} |\Delta_n f|^2}\Big]\ge \E_{n-1}\Big[f\Big(1-\frac{2}{5}|\Delta_n f|^2\Big)\Big(1+\Delta_n f+\frac{1}{2}|\Delta_n f|^2 + Q_n\Big)\Big].
	$$
	After expansion of the brackets, our inequality takes the following form:
	$$
	\E_{n-1}\Big[ f\Delta_n f + \Big( \frac{1}{2}-\frac{2}{5} \Big) f|\Delta_n f|^2 - \frac{2}{5} f(\Delta_n f)^3 - \frac{1}{5}f|\Delta_n f|^4 + fQ_n\Big(1-\frac{2}{5}|\Delta_n f|^2\Big) \ \Big]\ge 0.
	$$
	
	Since $|Q_n|\le\frac{1}{2}|\Delta_n f|^3$, we see that $\frac{1}{2}f|\Delta_n f|^2\ge \Big|fQ_n\Big(1-\frac{2}{5}|\Delta_n f|^2\Big)\Big|$. Also, we have $\E_{n-1}(f\Delta_n f)=\E_{n-1}|\Delta_n f|^2$ and hence it follows that this quantity dominates all remaining negative terms in our expression.
	
	\subsection{Proof of the inequality \eqref{main_ineq}}
	
	Let us now prove the inequality
	$$
	|\Delta_n (e^{\Delta_n f- \beta_n |\Delta_n f|^2} f)|\leq (4+e) |\Delta_n f|.
	$$
	We transform the left-hand side of this inequality using the formula \eqref{main_assump}:
	\begin{multline*}
		|\Delta_n (e^{\Delta_n f- \beta_n |\Delta_n f|^2} f)| = \Big| e^{\Delta_n f - \beta_n |\Delta_n f|^2}\E_n f - \E_{n-1}\Big(f\cdot e^{\Delta_n f - \beta_n |\Delta_n f|^2}\Big) \Big|\\
		=\Big|  e^{\Delta_n f - \beta_n |\Delta_n f|^2}\E_n f - \E_{n-1} f \Big|\\=\Big| e^{\Delta_n f - \beta_n |\Delta_n f|^2}\Delta_n f + \Big( e^{\Delta_n f - \beta_n |\Delta_n f|^2}- 1 \Big)\E_{n-1} f \Big|.
	\end{multline*}
The absolute value of the first summand does not exceed $e|\Delta_n f|$. Besides that, using that for $x\le 1$ we have $|e^x-1|\le 2|x|$, we see that the absolute value of the second summand is not greater than
$$
2|\Delta_n f - \beta_n |\Delta_n f|^2|\leq 2|\Delta_n f|+2\beta_n |\Delta_n f|^2\leq 4|\Delta_n f|,
$$
since $\beta_n \le \frac{1}{|\Delta_n f|}$.

\section{Proof of Theorem~\ref{mult_S}: application of Kislyakov's correction theorem}

\subsection{Formulation of general correction theorem}

Now we will formulate a version of the correction theorem from~\cite{Kis} which is relevant in our situation.

	Let $X$ be a Banach space of integrable functions on $\Omega$. For each function $g\in L^\infty$ we may define a functional $\Phi_g$ on $X$ as $\Phi_g(h)=\int_\Omega hg\, d\PP$. Suppose that the following two conditions hold.
	
	\begin{itemize}
		\item[A1.] Embedding $X\hookrightarrow L^1$ is continuous and the unit ball of $X$ is weakly compact in $L^1$.
		
		\item[A2.] For every $g\in L^\infty$ we have the following estimate:
		$$
		\PP\{|g|>t\}\leq \frac{c}{t}\|\Phi_g\|_{X^*}.
		$$
	\end{itemize}

\begin{theorem*}
	Let $f$ be a function in $L^\infty$, $\|f\|_\infty \le 1$. Then for every $0< \eps\le 1$ there exists a function $g$ in $X$ which satisfies the following conditions:
	\begin{gather}
		|g|+|f-g|=|f|;\\
		\PP\{f\neq g\}\le \eps\|f\|_1;\\
		\|g\|_X \le C(1+\log \eps^{-1}).
	\end{gather}
\end{theorem*}	

We may take $X=L^\infty\cap SL^\infty$, that is, the space of bounded functions such that their square functions are also bounded. The norm on $X$ is natural: $\|g\|_{X}=\max\{\|g\|_{\infty}, \|Sg\|_{\infty}\}$. 

Let us define the following operator that maps, say, functions in $L^2$ to $\ell^2$-valued functions: 
$$
T(f)=(\Delta_n f)_{n\ge 1}.
$$
Then $T$ is a linear operator and $\|Sf\|_{p}=\|Tf\|_{L^p(\ell^2)}$. In particular, the norm on $X$ can be rewritten as $\|g\|_X=\max\{\|g\|_\infty, \|Tg\|_{L^\infty (\ell^2)}\}$. It is also easy to see that $X$ is a Banach space.

It turns out that if $X$ satisfies condition A1 then it is a dual space. Let $Y$ be the closure of the set $\{\Phi_g: g\in L^\infty\}$ in the norm of $X^*$. Then following lemma is proved in \cite{Kis}.

\begin{lemma}
	Assume that $X$ satisfies condition A1. Then the dual of $Y$ is $X$ and the weak* topology on the ball of $X$ coincides with the weak topology of $L^1$.\label{lemma-kis}
\end{lemma}

It remains only to verify conditions A1 and A2 for $X$: then Theorem~\ref{mult_S} will be proved. This can be done similarly to \cite[Lemma 2]{Kis}. However, here we have an operator $T$ instead of singular integral operators, so we will briefly discuss some details which are different in our situation.

\subsection{Verification of conditions of general theorem}

For a Banach space $Y$ we will denote by $B_Y$ the unit ball in this space.

\subsubsection{Verification of condition A1.}

The continuity of embedding $X\hookrightarrow L^1$ is obvious, so we only need to show the weak compactness of $B_X$ in $L^1$. Suppose that $\|f_n\|_\infty \le 1$, $\|S(f_n)\|_\infty \le 1$ and $f_n \to f$ weakly in $L^1$. We have to verify that $\|S(f)\|_\infty \le 1$. Note that on $B_{L^\infty}$ the weak convergence in $L^1$ is the same as weak* convergence in $L^\infty$.

Consider the operator $U:L^2(\ell^2)\to L^2$ that acts on $v=(v_m)_{m\ge 1}\in L^2(\ell^2)$ in a following way:
\begin{equation}
Uv = \sum_{m\ge 1}\Delta_m v_m.\label{dual-op}
\end{equation}
An easy formal verification shows that $U^*=T$.

Next, we also see that $U$ is a bounded operator from $H^1(\ell^2)$ to $L^1$ (for the definition of martingale Hardy space $H^1$ and its most important properties, including duality between $H^1$ and $\BMO$, see for example \cite[Chapter 2]{Long}; we need here $\ell^2$-valued Hardy spaces but there is no difference: the square function, maximal function etc. can be defined similarly to scalar-valued case and they have the same properties. It is also worth mentioning that a lot of information concerning vector-valued martingales in a much more general context of UMD Banach spaces can be found in the book \cite{Pis}). Indeed, we have:
\begin{multline*}
\|v\|_{H^1(\ell^2)}=\Big\| \Big( \sum_{k\ge 1} |\Delta_k v|_{\ell^2}^2 \Big)^{1/2} \Big\|_1=\Big\| \Big( \sum_{k\ge 1}\sum_{m\ge 1}|\Delta_k v_m|^2 \Big)^{1.2} \Big\|_1\\
\ge \Big\| \Big( \sum_{k\ge 1}|\Delta_k v_k|^2 \Big)^{1/2} \Big\|_1=\|S(Uv)\|_1 \gtrsim \|Uv\|_1.
\end{multline*}

Now we see that $T$ is a bounded and weakly* continuous operator from $L^\infty$ to $\BMO(\ell^2)=(H^1(\ell^2))^*$. The space $\BMO(\ell^2)$ is defined by the following seminorm:
$$
\|v\|_{\BMO(\ell^2)}=\sup_{n\ge 1}\|\E_n(|v-\E_{n-1}v|_{\ell^2}^2)\|_\infty.
$$
Therefore, $Tf_n\to Tf$ weakly* in $\BMO(\ell^2)$. Besides that, $Tf_n\in B_{L^\infty_0(\ell^2)}$ where $L^\infty_0(\ell^2)$ denotes the space of functions $v\in L^\infty(\ell^2)$ such that $\E v = 0\in\ell^2$. But this ball is compact in the weak* topology of $\BMO$ since the inclusion $L^\infty_0(\ell^2)\hookrightarrow \BMO(\ell^2)$ is weak* continuous (since it is a dual operator to the inclusion $H_0^1(\ell^2)\hookrightarrow L_0^1(\ell^2)$ where the subscript $0$ is again used to indicate the subspace which consists of functions with mean value $0$). The condition A1 is verified.

\subsubsection{Verification of condition A2} 

The main ingredient for verification of condition A2 is the weak type $(1,1)$ estimate for the operator $U$.

\begin{lemma}
	The operator $U$ defined by \eqref{dual-op} is a bounded operator from $L^1(\ell^2)$ to $L^{1,\infty}$.\label{weak-type}
\end{lemma}

\begin{proof}
	For arbitrary $v\in L^1(\ell^2)$ and $t>0$ we need to prove the inequality
	\begin{equation}
	\PP \Big\{ \Big| \sum_{k\ge 1}\Delta_k v_k \Big|>t \Big\}\leq \frac{c}{t} \|v\|_{L^1(\ell^2)}.\label{weak}
	\end{equation}
	
	This inequality is an immediate consequence of Gundy's decomposition for ($\ell^2$-valued) martingales. For the proof of Gundy's decomposition see \cite[Theorem 1.3.1.8]{Long}; for the vector-valued version see also \cite[Theorem 5.15]{Pis}. 
	
	For arbitrary $t>0$ we decompose $v\in L^1(\ell^2)$ as $v=W+Y+Z$ where
	\begin{gather}
		\|W\|_{L^2(\ell^2)}^2 \leq 2t \|v\|_{L^1(\ell^2)};\\
		\Big\| \sum_{k\ge 1} |\Delta_k Y|_{\ell^2} \Big\|_1 \leq 4 \|v\|_{L^1(\ell^2)};\\
		|\{\sup_k |\Delta_k Z|_{\ell^2}>0\}|\leq \frac{\|v\|_{L^1(\ell^2)}}{t}.
	\end{gather}
	We should estimate each of three summands. For the first one, we have
	$$
	\PP \Big\{ \Big| \sum_{k\ge 1}\Delta_k W_k \Big|>t \Big\} \leq \frac{\|W\|_{L^2(\ell^2)}^2}{t^2}\leq \frac{2\|v\|_{L^1(\ell^2)}}{t}.
	$$
	For the second, we write:
	\begin{multline*}
	\PP \Big\{ \Big| \sum_{k\ge 1}\Delta_k Y_k \Big|>t \Big\}\leq \frac{1}{t} \Big\| \sum_{k\ge 1} |\Delta_k Y_k| \Big\|_1\\ \leq \frac{1}{t} \Big\|\sum_{k\ge 1} \Big( \sum_{m\ge 1} |\Delta_k Y_m|^2 \Big)^{1/2} \Big\|_1\leq \frac{4\|v\|_{L^1(\ell^2)}}{t}.
	\end{multline*}
	Finally, for the last one we have:
	$$
	\PP \Big\{ \Big| \sum_{k\ge 1}\Delta_k Z_k \Big|>t \Big\}\leq \PP \Big\{ \Big| \sum_{k\ge 1}\Delta_k Z_k \Big|>0 \Big\}\leq \frac{\|v\|_{L^1(\ell^2)}}{t}.
	$$
	Combining these estimates, we get the inequality \eqref{weak} (with $3t$ instead of $t$ in the left-hand side).
	
\end{proof}

From our verification of condition A1 (essentially taken from \cite{Kis}) and Lemma~\ref{lemma-kis} it is not difficult to conclude that the natural embedding
$$
X\hookrightarrow (L^\infty \oplus L^\infty(\ell^2))_\infty,\quad f\mapsto (f, Tf)
$$
is a weak* continuous isometry. For every $g$ the functional $\Phi_g$ is in predual of $X$ and therefore we can use Hahn--Banach theorem in order to extend it to a weak* continuous functional on $(L^\infty \oplus L^\infty(\ell^2))_\infty$. It means that there exist functions $u\in L^1$, $v\in L^1(\ell^2)$ such that 
\begin{equation}
\|u\|_1 + \|v\|_{L^1(\ell^2)}\leq \|\Phi_g\|_{X^*}\label{hahn-ban}
\end{equation}
and for every $f\in X$
$$
\int_\Omega fg\, d\PP = \langle \Phi_g, f \rangle = \langle f, u \rangle + \langle Tf, v \rangle = \int_\Omega(fu)\, d\PP + \int_\Omega \sum_{m\ge 1} v_m \Delta_m f\, d\PP.
$$
We claim that $g=u + Uv$. Indeed, fix some $N>0$. We may apply the above identity to any function $f\in L^\infty(\Omega, \F_N, \PP)$ (this function is clearly in $X$ because there are only finitely many non-zero summands in the definition of the expression $Sf$) to get that
$$
\int_\Omega f\E_N g\, d\PP = \int_\Omega f((\E_N u) + U(\E_N v))\, d\PP.
$$
It follows that $\E_N g = \E_N u + U(\E_N v)$. It remains to tend $N$ to infinity (then $\E_N v \to v$ in $L^1(\ell^2)$ and hence by Lemma~\ref{weak-type} $U(\E_N v) \to Uv$ in $L^{1,\infty}$) and we indeed get that $g=u + Uv$.

Condition A2 now follows from \eqref{hahn-ban} if we apply Lemma~\ref{weak-type} once again.

\section{Open problems}

\textbf{Problem 1}. As we have seen, Theorem~\ref{mult_S} is non-constructive but slightly more powerful than Theorem~\ref{mult_S_const}. We do not know however if there exists a general correction theorem in the spirit of Theorem~\ref{mult_S} for the operator $\sigma$. The main obstacle is that for application of Kislyakov's theorem we need a \emph{linear} operator; it is easy to linearize the operator $S$ (we have done it above by defining the operator $T$) but it is not clear if something similar can be done for $\sigma$. Therefore, it would be interesting to obtain some version of correction theorem similar to Theorem~\ref{mult_S} for conditional square function operator $\sigma$.

\textbf{Problem 2.} In the paper \cite{J-M} in the dyadic case it is noted that the constructed multiplier function $m$ lies in the (dyadic) Muckenhoupt class $A_\infty$. Therefore, the natural question arises: can the abstract approach which we used in the proof of Theorem~\ref{mult_S} be applied in order to obtain a multiplier function $m$ which lies in the class $A_\infty$ (even in the simplest dyadic case)? The same question is also interesting for other correction theorems, say, from the paper \cite{Kis}.

\textbf{Problem 3}. There exists a notion of martingale $A_\infty$ class related to an arbitrary filtered probability space, see e.g. \cite[Chapter 6]{Long}. We did not investigate the properties of the multiplier functions which were constructed in the proofs of Theorems~\ref{mult_sigma} and~\ref{mult_S_const} but the question if they belong to a (properly defined) $A_\infty$ class seems to be interesting and non-trivial.

\medskip

\textbf{Acknowledgments.} The author is kindly grateful to Professor P. F. X. M\"{u}ller for many fruitful discussions and useful comments concerning the material of the present paper.

\end{document}